\documentclass{amsart}

\usepackage[margin=1.5in]{geometry} 
\usepackage[english]{babel} 
\usepackage[utf8]{inputenc} 
\usepackage[T1]{fontenc} 
\usepackage{lmodern} 
\usepackage{exscale} 
\usepackage{microtype} 

\usepackage{latexsym,amsxtra,amscd,amsmath,amssymb,amsthm,stmaryrd,mathtools,leftindex} 
\usepackage{bbold,bbm} 

\usepackage[dvipsnames]{xcolor} 
\usepackage{graphics} 
\usepackage{longtable} 
\usepackage{colortbl, booktabs} 
\usepackage{multicol, multirow} 

\usepackage{hyperref} 
\hypersetup{
    colorlinks=true,
    linkcolor=blue,
    filecolor=blue,      
    urlcolor=cyan,
    linktocpage=true
}

\usepackage{enumitem} 
\usepackage{comment} 
\usepackage[normalem]{ulem} 

\usepackage[clock]{ifsym} 

\usepackage{amsfonts} 

\usepackage[all]{xy} 

\makeatletter
\@namedef{subjclassname@2020}{%
  \textup{2020} Mathematics Subject Classification}
 
\makeatother

\newcommand{\stkout}[1]{\ifmmode\text{\sout{\ensuremath{#1}}}\else\sout{#1}\fi}

\newtheorem{theorem}{Theorem}[section]

\newtheorem{proposition}[theorem]{Proposition}

\theoremstyle{definition}
\newtheorem{definition}[theorem]{Definition}
\newtheorem{remark}[theorem]{Remark}

\newtheorem{example}[theorem]{Example}

\theoremstyle{remark}

\numberwithin{equation}{section}

\begin{document}

\title{On the annihilator condition for skew PBW extensions}
\author{Karol Herrera}
\address{Universidad ECCI}
\curraddr{Campus Universitario}
\email{kherrerac@ecci.edu.co}
\thanks{}


\author{Sebastián Higuera}
\address{Universidad ECCI}
\curraddr{Campus Universitario}
\email{shiguerar@ecci.edu.co}
\thanks{}


\author{Andr\'es Rubiano}
\address{Universidad ECCI}
\curraddr{Campus Universitario}
\email{arubianos@ecci.edu.co}
\thanks{}

\dedicatory{Dedicated to our dear friend  Armando Reyes}

\begin{abstract} 
We investigate the annihilator condition $(a.c.)$ for skew Poincaré-Birkhoff-Witt extensions. We prove that some results about the annihilator condition $(a.c.)$ for skew polynomial rings also hold for skew PBW extensions. We also demonstrate that the behavior of annihilators is determined by the defining relations of the extension. Our results extend those corresponding presented for skew polynomial rings
\end{abstract}

\subjclass[2020]{16D25, 16S30, 16S32, 16S36, 16S38}
\keywords{Annihilator condition, compatible  ring, Armendariz ring, p.q.-Baer ring, skew polynomial ring, skew PBW extension.}

\maketitle	



\section{Introduction}

Throughout the paper, every ring is associative (not necessarily commutative) with identity unless otherwise stated. The center and the set of all idempotent elements of $R$ is denoted by \( C(R) \) and \( I(R) \), respectively. In addition, $P(R)$, $\mathrm{nil}^{*}(R)$ and $\mathrm{nil}(R)$ denote the prime radical of $R$, the upper nilradical of $R$ (that is, the sum of all nil ideals of $R$) and the set of nilpotent elements of $R$, respectively. If $P(R) = \mathrm{nil}(R)$ then $R$ is called \emph{2-primal}, and if $\mathrm{nil}^{*}(R) = \mathrm{nil}(R)$ then $R$ is an \emph{NI} ring. An idempotent element $e \in R$ is called \textit{left semi-central} (resp., \textit{right semi-central}) if $re = ere$ (resp., $er = ere$) for all $r \in R$. The sets of left and right semi-central idempotents of \( R \) will be written as \( S_l(R) \) and \( S_r(R) \). If $S\subseteq R$, we define the \textit{right annihilator} and the \textit{left annihilator} as $r_R(S) = \{a \in R \mid Sa = 0\}$ and $\ell_R(S) = \{a \in R \mid aS = 0\}$, respectively.

Jerison~\cite{HenriksenJerison} introduced the concept of the \textit{annihilator condition} \((a.c.)\) for commutative reduced rings (\(R\) is called \textit{reduced} if it has no non-zero nilpotent elements). In the commutative context, \( R \) satisfies the annihilator condition \((a.c.)\) (or has Property \((a.c.)\)) if the annihilator of each finitely generated ideal \( I \) of \( R \) coincides with the annihilator of a single element of \( R \). Later, Lucas~\cite{Lucas} generalized this definition to all commutative rings. The class of rings satisfying Property \((a.c.)\) is quite extensive. Some examples include subdirect sums of totally ordered integral domains, Bézout rings, direct sums of rings with Property \((a.c.)\), reduced Noetherian rings, as well as certain nontrivial reduced graded rings (see~\cite{HenriksenJerison,Lucas}).

Hong et al.~\cite{Hong2009} extended this notion to the noncommutative case. A ring \( R \) is said to satisfy \textit{Property \((a.c.)\) on the right} (respectively, \textit{on the left}) if for every finitely generated right ideal \( I \) of \( R \), there exists some \( c \in R \) such that $r_R(I) = r_R(cR)$ (respectively, $\ell_R(I) = \ell_R(Rc)$). They investigated many classes of rings that enjoy Property \((a.c.)\), proving that this includes subdirect products of fully ordered semiprime rings, semiprime rings with many minimal prime ideals, semiprime rings with the ascending chain condition on annihilators, and others. They showed that matrix rings over rings with Property \((a.c.)\) also satisfy the same property \cite[Theorem 2.1]{Hong2009}, and that the skew polynomial rings \( R[x;\sigma,\delta] \) introduced by Ore \cite{Ore1933}, with \( \sigma \) an automorphism of $R$ and \( \delta \) a $\sigma$-derivation of $R$ has Property \((a.c.)\) \cite[Theorem 1.10]{Hong2009}. 

Following Rege and Chhawchharia~\cite{Rege1997}, a ring \(R\) is \emph{Armendariz} if for any two polynomials in \(R[x]\), 
whenever their product is zero, then the product of each pair of their coefficients is also zero. Hirano~\cite{Hirano} introduced the \emph{quasi-Armendariz rings}, 
which generalize Armendariz rings. This class of rings satisfies the following condition: 
for any two polynomials \(p(x)=\sum_{i=0}^{n}p_{i}x^{i}\), \(q(x)=\sum_{j=0}^{m}q_{j}x^{j}\in R[x]\), 
if \(p(x)R[x]q(x)=0\), then \(p_{i}Rq_{j}=0\) for all \(i,j\).  It is well-known that reduced rings are Armendariz, and Armendariz rings are quasi-Armendariz rings, but the converse are not true in general \cite[Proposition 2.1]{Rege1997}. Hong et al.~\cite{Hong2010} extended this definition to the case of skew polynomial rings as follows: a ring \(R\) is called \textit{\(\sigma\)-skew quasi-Armendariz} if for \(p(x)=\sum_{i=0}^{n}p_{i}x^{i}\), \(q(x)=\sum_{j=0}^{m}q_{j}x^{j}\in R[x,\sigma]\), whenever \(p(x)R[x,\sigma]q(x)=0\), it follows that \(p_{i}R\sigma^{i}(q_{j})=0\) for all \(i,j\). The Armendariz rings over skew polynomial rings has been widely studied by several authors (see \cite{AndersonCamillo1998,Armendariz1974,HongKimKwak2003}).

If $\sigma$ is an endomorphism of $R$ and we have that $r\sigma(r)=0$ implies that $r=0$ for all $r\in R$, then $R$ is {\em rigid}, and if there exists a rigid endomorphism $\sigma$ of $R$, then $R$ is $\sigma$-{\em rigid}. Annin in his Ph.D. Thesis \cite{Annin2002PhD} (see \cite{Annin2004, HashemiMoussavi2005}) or Hashemi and Moussavi \cite{HashemiMoussavi2005} introduced compatible rings as a generalization of $\sigma$-rigid rings and reduced rings. $R$ is {\it $\sigma$-compatible} if for each $a, b \in R$, then $ab = 0$ if and only if $a\sigma(b) = 0$; $R$ is {\it $\delta$-compatible} if for each $a, b \in R$, we obtain that $ab = 0$ implies that $a\delta(b) = 0$, and if $R$ is both $\sigma$-compatible and $\delta$-compatible, then $R$ is called {\it $(\sigma,\delta)$-compatible}. 

Malki and Louzari \cite{MalkiLouzari2024} investigated the relationship between Armendariz properties and the annihilator condition $(a.c.)$ on skew polynomial rings $R[x;\sigma]$. They proved that if $R$ is $\sigma$-skew quasi-Armendariz rings always satisfy Property $(a.c.)$ on the left, but not necessarily on the right \cite[Proposition 3.2 and Example 3.4]{MalkiLouzari2024}. The investigation established sufficient conditions for these rings to also possess Property $(a.c.)$, particularly when $R$ is $\sigma$-compatible or a semiprime ring with $\sigma$ surjective \cite[Corollary 3.9]{MalkiLouzari2024}. They studied how Property \((a.c.)\) passes from \(R\) to its skew polynomial ring \(R[x;\sigma]\) \cite[Proposition 3.10]{MalkiLouzari2024}, giving new insight into the behavior of this property on skew polynomial rings.

Thinking about the above results on the annihilator condition and Armendariz properties for skew polynomial rings, a natural task is how these ideas can be extended to more general noncommutative contexts. Our purpose in this paper is to study the connection between the annihilator condition and Armendariz properties on {\em skew Poincaré–Birkhoff–Witt extensions} ({\em SPBW extensions}, for short), which are a large class of noncommutative polynomial rings that include many important examples. The SPBW extension were defined y Gallego and Lezama \cite{LezamaGallego2011} as a generalization of the PBW extensions considered by Bell and Goodearl \cite{BellGoodearl1988} and skew polynomial rings of injective type introduced by Ore \cite{Ore1933}. Several authors have shown that SPBW extensions also generalize families of noncommutative algebras, such as 3-dimensional skew polynomial algebras introduced by Bell and Smith \cite{BellSmith1990}, diffusion algebras considered by Isaev {et al.} \cite{IsaevPyatovRittenberg2001}, ambiskew polynomial rings introduced by Jordan (see \cite{Jordan1993} and references therein), almost normalizing extensions defined by McConnell and Robson \cite{McConnellRobson2001} and skew bi-quadratic algebras recently introduced by Bavula \cite{Bavula2023}. For more details on the SPBW extensions, see \cite{Fajardoetal2020, HigueraRamirezReyes2024, HigueraReyes2022}.

The paper is organized as follows. In Section \ref{Definitions}, we recall fundamental definitions and establish key properties of SPBW extensions and $(\Sigma, \Delta)$-compatible rings.
Section \ref{MainResults} contains our main theorems: Theorems \ref{BaerSA1Theorem}, \ref{AbelianNIppTheorem}, \ref{BijectiveSPBWppTheorem}, 
and \ref{pqBaerTheorem}, which establish conditions 
under which skew PBW extensions satisfy Property $(a.c.)$ on the right (resp., left), 
generalizing previous work on skew polynomial rings $R[x;\sigma]$ by Malki and Louzari \cite{MalkiLouzari2024}. 
Section \ref{Examplespaper} illustrates our findings with examples of noncommutative algebras that are not skew polynomial rings.


Throughout the paper, $\mathbb{N}$, $\mathbb{Z}$, $\mathbb{R}$, and $\mathbb{C}$ denote the standard numerical systems, with $\mathbb{N}$ including the zero element. The symbol $\Bbbk$ denotes a field, and $\Bbbk^* := \Bbbk \setminus \{0\}$.

\section{Definitions and elementary properties}\label{Definitions}

\begin{definition}[{\cite[Definition 1]{LezamaGallego2011}}] \label{gpbwextension}
A ring $A$ is called a {\em skew PBW} ({\em SPBW}) {\em extension} {\em of} $R$, which is denoted by $A:=\sigma(R)\langle
x_1,\dots,x_n\rangle$, if the following conditions hold:
\begin{enumerate}
\item[\rm (i)]$R$ is a subring of $A$ sharing the same identity element.

\item[\rm (ii)] There exist finitely many elements $x_1,\dots ,x_n\in A$ such that $A$ is a left free $R$-module, with basis the set of standard monomials
\begin{center}
${\rm Mon}(A):= \{x^{\alpha}:=x_1^{\alpha_1}\cdots
x_n^{\alpha_n}\mid \alpha=(\alpha_1,\dots ,\alpha_n)\in
\mathbb{N}^n\}$.
\end{center}

We consider $x^0_1\cdots x^0_n := 1 \in {\rm Mon}(A)$.
\item[\rm (iii)] For every $1\leq i\leq n$ and any non-zero element $r\in R$, there exists a non-zero element $c_{i,r}\in R$ such that $x_ir-c_{i,r}x_i\in R$.
\item[\rm (iv)] For $1\leq i,j\leq n$, there exists a non-zero element $d_{i,j}\in R$ such that
\[
x_jx_i-d_{i,j}x_ix_j\in R+Rx_1+\cdots +Rx_n,
\]

i.e. there exist elements $r_0^{(i,j)}, r_1^{(i,j)}, \dotsc, r_n^{(i,j)} \in R$ with
\begin{center}
$x_jx_i - d_{i,j}x_ix_j = r_0^{(i,j)} + \sum_{k=1}^{n} r_k^{(i,j)}x_k$.    
\end{center}
\end{enumerate}
\end{definition}

Since ${\rm Mon}(A)$ is a left $R$-basis of $A$, this implies that the elements $c_{i,r}$ and $d_{i, j}$ are unique, and every non-zero element $f \in A$ can be uniquely expressed as $f = a_0 + a_1X_1 + \cdots + a_mX_m$ with $a_i \in R$, $X_0=1$, and $X_i \in \text{Mon}(A)$, for $0 \leq i \leq m$ \cite[Remark 2]{LezamaGallego2011}. 

If $A=\sigma(R)\langle x_1,\dots,x_n\rangle$ is an SPBW extension of $R$, there exist an injective endomorphism $\sigma_i$ of $R$ and a $\sigma_i$-derivation $\delta_i$ of $R$ such that $x_ir=\sigma_i(r)x_i+\delta_i(r)$, for all $1\leq i\leq n$ and $r\in R$ \cite[Proposition 3]{LezamaGallego2011}. The notation $\Sigma:=\{\sigma_1,\dots,\sigma_n\}$ and $\Delta:=\{\delta_1,\dots,\delta_n\}$ denotes the corresponding sets of injective endomorphisms and $\sigma_i$-derivations, respectively.

\begin{definition}\label{quasicommutative}
If $A=\sigma(R)\langle x_1,\dots,x_n\rangle$ is an SPBW extension of $R$, then:
\begin{itemize}
    \item[{\rm (i)}] {\cite[Definition 4]{LezamaGallego2011}} $A$ is {\it quasi-commutative} if the conditions ${\rm (iii)}$ and ${\rm (iv)}$ formulated in Definition \ref{gpbwextension} are replaced by the following: 
\begin{enumerate}
    \item[(iii')] For every $1 \leq i \leq n$ and non-zero element $r \in R$ there exists a non-zero element $c_{i,r} \in R$ such that $x_ir = c_{i,r}x_i$.
    
\item[(iv')] For every $1 \leq i, j \leq n$ there exists a non-zero element $d_{i,j} \in R$ such that $x_jx_i = d_{i,j}x_ix_j$.    
\end{enumerate}

    \item[{\rm (ii)}] \cite[Definition 4]{LezamaGallego2011} $A$ is {\it bijective} if  $\sigma_i$ is bijective for all $1 \leq i \leq n$, and $d_{i,j}$ is invertible for every $1 \leq i <j \leq n$.
    \item [\rm (iii)] \cite[Definition 2.3]{AcostaLezamaReyes2015} If $\sigma_i={\rm id}_R$ is the identity map of $R$ for all $1\le i \le n$, then $A$ is called of \textit{derivation type}. If $\delta_i=0$ for all $1\le i \le n$, then $A$ is of \textit{endomorphism type}.
\end{itemize}
\end{definition}

\begin{remark}[{\cite[Section 3]{LezamaGallego2011}}]\label{definitioncoefficients}
If $A=\sigma(R)\langle x_1,\dots,x_n\rangle$ is an SPBW extension of $R$, then: 
\begin{enumerate}
\item[\rm (1)] For any non-zero element $\alpha=(\alpha_1,\dots,\alpha_n)\in \mathbb{N}^n$, we will write 
$\sigma^{\alpha}:=\sigma_1^{\alpha_1}\circ \dotsb \circ \sigma_n^{\alpha_n}$, $\delta^{\alpha} = \delta_1^{\alpha_1} \circ \dotsb \circ \delta_n^{\alpha_n}$, where $\circ$ represents the usual function composition and 
$|\alpha|:=\alpha_1+\cdots+\alpha_n$. If
$\beta=(\beta_1,\dots,\beta_n)\in \mathbb{N}^n$, then
$\alpha+\beta:=(\alpha_1+\beta_1,\dots,\alpha_n+\beta_n)$.

\item[\rm (2)] Let $\succeq$ be a total order defined on ${\rm Mon}(A)$. If $x^{\alpha}\succeq x^{\beta}$ but $x^{\alpha}\neq x^{\beta}$, we write $x^{\alpha}\succ x^{\beta}$. If $f$ is a non-zero element of $A$, then we use expressions as $f=\sum_{i=0}^ma_iX_i$, with $a_i\in R$, and $X_m\succ \dotsb \succ X_1$. With this notation, we define ${\rm
lm}(f):=X_m$, the \textit{leading monomial} of $f$; ${\rm
lc}(f):=a_m$, the \textit{leading coefficient} of $f$; ${\rm
lt}(f):=a_mX_m$, the \textit{leading term} of $f$; ${\rm exp}(f):={\rm exp}(X_m)$, the \textit{order} of $f$. Notice that $\deg(f):={\rm max}\{\deg(X_i)\}_{i=1}^m$. Finally, if $f=0$, then
${\rm lm}(0):=0$, ${\rm lc}(0):=0$, ${\rm lt}(0):=0$. We also
consider $X\succ 0$ for any $X\in {\rm Mon}(A)$. Thus, we extend $\succeq$ to ${\rm Mon}(A)\cup \{0\}$.
\end{enumerate}
\end{remark}

\begin{proposition}[{\cite[Theorem 7]{LezamaGallego2011}}]\label{coefficientes}
If $A$ is a polynomial ring with respect to the set of indeterminates $\{x_1,\dots,x_n\}$ and $R$ is the coefficient ring, then $A$ is an SPBW  extension of $R$ if and only if the following conditions hold:
\begin{enumerate}
\item[\rm (1)] For each $x^{\alpha}\in {\rm Mon}(A)$ and every non-zero element $r\in R$, there exist unique elements $r_{\alpha}:=\sigma^{\alpha}(r)\in R\ \backslash\ \{0\}$ and $p_{\alpha ,r}\in A$ such that $x^{\alpha}r=r_{\alpha}x^{\alpha}+p_{\alpha, r}$,  where $p_{\alpha ,r}=0$, or $\deg(p_{\alpha ,r})<|\alpha|$ if
$p_{\alpha , r}\neq 0$. If $r$ is left invertible,  so is
$r_\alpha$.

\item[\rm (2)]For each $x^{\alpha},x^{\beta}\in {\rm Mon}(A)$,  there exist unique elements $c_{\alpha,\beta}\in R\ \backslash\ \{0\}$ and $p_{\alpha,\beta}\in A$ such that $x^{\alpha}x^{\beta}=c_{\alpha,\beta}x^{\alpha+\beta}+p_{\alpha,\beta}$, where $c_{\alpha,\beta}$ is left invertible, $p_{\alpha,\beta}=0$, or $\deg(p_{\alpha,\beta})<|\alpha+\beta|$ if
$p_{\alpha,\beta}\neq 0$.
\end{enumerate}
\end{proposition}

\begin{proposition}[{\cite[Proposition 2.9 and Remark 2.10 iv)]{Reyes2015Baer}}] \label{Reyes2015Proposition2.9} If $\sigma(R)\langle x_1,\dots,x_n\rangle$ is an SPBW extension of $R$, then
\begin{itemize}
    \item[{\rm (1)}] If $\alpha=(\alpha_1,\dotsc, \alpha_n)\in \mathbb{N}^{n}$ and $r\in R$ with $r\neq 0$, then  
{\footnotesize{\begin{align*}
x^{\alpha}r = &\ x_1^{\alpha_1}x_2^{\alpha_2}\dotsb x_{n-1}^{\alpha_{n-1}}x_n^{\alpha_n}r = x_1^{\alpha_1}\dotsb x_{n-1}^{\alpha_{n-1}}\biggl(\sum_{j=1}^{\alpha_n}x_n^{\alpha_{n}-j}\delta_n(\sigma_n^{j-1}(r))x_n^{j-1}\biggr)\\
&\ + x_1^{\alpha_1}\dotsb x_{n-2}^{\alpha_{n-2}}\biggl(\sum_{j=1}^{\alpha_{n-1}}x_{n-1}^{\alpha_{n-1}-j}\delta_{n-1}(\sigma_{n-1}^{j-1}(\sigma_n^{\alpha_n}(r)))x_{n-1}^{j-1}\biggr)x_n^{\alpha_n}\\
&\ + x_1^{\alpha_1}\dotsb x_{n-3}^{\alpha_{n-3}}\biggl(\sum_{j=1}^{\alpha_{n-2}} x_{n-2}^{\alpha_{n-2}-j}\delta_{n-2}(\sigma_{n-2}^{j-1}(\sigma_{n-1}^{\alpha_{n-1}}(\sigma_n^{\alpha_n}(r))))x_{n-2}^{j-1}\biggr)x_{n-1}^{\alpha_{n-1}}x_n^{\alpha_n}\\
&\ + \dotsb + x_1^{\alpha_1}\biggl( \sum_{j=1}^{\alpha_2}x_2^{\alpha_2-j}\delta_2(\sigma_2^{j-1}(\sigma_3^{\alpha_3}(\sigma_4^{\alpha_4}(\dotsb (\sigma_n^{\alpha_n}(r))))))x_2^{j-1}\biggr)x_3^{\alpha_3}x_4^{\alpha_4}\dotsb x_{n-1}^{\alpha_{n-1}}x_n^{\alpha_n} \\
&\ + \sigma_1^{\alpha_1}(\sigma_2^{\alpha_2}(\dotsb (\sigma_n^{\alpha_n}(r))))x_1^{\alpha_1}\dotsb x_n^{\alpha_n}, \ \ \ \ \ \ \ \ \ \ \sigma_j^{0}:={\rm id}_R\ \ {\rm for}\ \ 1\le j\le n.
\end{align*}}}
    
\item[{\rm (2)}] If $a_i, b_j\in R$ and $X_i:=x_1^{\alpha_{i1}}\dotsb x_n^{\alpha_{in}},\ Y_j:=x_1^{\beta_{j1}}\dotsb x_n^{\beta_{jn}}$, when we compute every summand of $a_iX_ib_jY_j$ we obtain products of the coefficient $a_i$ with several evaluations of $b_j$ in $\sigma$'s and $\delta$'s depending on the coordinates of $\alpha_i$. This assertion follows from the expression:
\begin{align*}
a_iX_ib_jY_j = &\ a_i\sigma^{\alpha_{i}}(b_j)x^{\alpha_i}x^{\beta_j} + a_ip_{\alpha_{i1}, \sigma_{i2}^{\alpha_{i2}}(\dotsb (\sigma_{in}^{\alpha_{in}}(b_j)))} x_2^{\alpha_{i2}}\dotsb x_n^{\alpha_{in}}x^{\beta_j} \\
&\ + a_i x_1^{\alpha_{i1}}p_{\alpha_{i2}, \sigma_3^{\alpha_{i3}}(\dotsb (\sigma_{{in}}^{\alpha_{in}}(b_j)))} x_3^{\alpha_{i3}}\dotsb x_n^{\alpha_{in}}x^{\beta_j} \\
&\ + a_i x_1^{\alpha_{i1}}x_2^{\alpha_{i2}}p_{\alpha_{i3}, \sigma_{i4}^{\alpha_{i4}} (\dotsb (\sigma_{in}^{\alpha_{in}}(b_j)))} x_4^{\alpha_{i4}}\dotsb x_n^{\alpha_{in}}x^{\beta_j}\\
&\ + \dotsb + a_i x_1^{\alpha_{i1}}x_2^{\alpha_{i2}} \dotsb x_{i(n-2)}^{\alpha_{i(n-2)}}p_{\alpha_{i(n-1)}, \sigma_{in}^{\alpha_{in}}(b_j)}x_n^{\alpha_{in}}x^{\beta_j} \\
&\ + a_i x_1^{\alpha_{i1}}\dotsb x_{i(n-1)}^{\alpha_{i(n-1)}}p_{\alpha_{in}, b_j}x^{\beta_j}.
\end{align*}
\end{itemize}
\end{proposition}

In the setting of SPBW extensions, Reyes and Su\'arez \cite{ReyesSuarez2018-3} defined the $\Sigma$-rigidness condition as follows: If $\Sigma=\{\sigma_1,\dots,\sigma_n\}$ is a family of endomorphisms of $R$, then $\Sigma$ is called a \emph{rigid endomorphisms family} if $a\sigma^{\alpha}(a)= 0$ implies that $a = 0$, where $a\in R$ and $\alpha\in \mathbb{N}^n$, and $R$ is a $\Sigma$-\emph{rigid} if there exists a rigid endomorphisms family $\Sigma$ of $R$ \cite[Definition 3.1]{ReyesSuarez2018-3}. Hashemi et al. \cite{HashemiKhalilAlhevaz2017}, and Reyes and Su\'arez \cite {ReyesSuarez2018UMA} introduced independently $(\Sigma, \Delta)$-{\em compatible rings} as a generalization of $(\sigma, \delta)$-compatible rings. Examples, ring and module theoretic properties of these objects have been investigated \cite{HashemiKhalilAlhevaz2017, HashemiKhalilAlhevaz2019, LouzariReyes2020, OuyangBirkenmeier2012, ReyesSuarez2018UMA}.

\begin{definition}[{\cite[Definition 3.1]{HashemiKhalilAlhevaz2017}}; {\cite[Definition 3.2]{ReyesSuarez2018UMA}}]\label{Definition3.52008}
$R$ is called {\it $\Sigma$-compatible} if for all $a, b \in R$, we obtain that $a\sigma^{\alpha}(b) = 0$ if and only if $ab = 0$, where $\alpha \in \mathbb{N}^n$; $R$ is {\it $\Delta$-compatible} if for all $a, b \in R$, we have that $ab = 0$ implies that $a\delta^{\beta}(b)=0$, where $\beta \in \mathbb{N}^n$. If $R$ is both $\Sigma$-compatible and $\Delta$-compatible, then $R$ is called {\it $(\Sigma, \Delta)$-compatible}.
\end{definition}

The following proposition is the natural generalization of \cite[Lemma 2.1]{HashemiMoussavi2005}.

\begin{proposition}[{\cite[Proposition 3.8]{ReyesSuarez2018UMA}}] \label{colosss}
If $R$ is $(\Sigma,\Delta)$-compatible, then for every $a, b
\in R$, the following assertions hold:
\begin{enumerate}
\item [\rm (1)] if $ab=0$, then $a\sigma^{\theta}(b) = \sigma^{\theta}(a)b=0$, for each $\theta\in \mathbb{N}^{n}$.

\item [\rm (2)] If $\sigma^{\beta}(a)b=0$ for some $\beta\in \mathbb{N}^{n}$, then $ab=0$.

\item [\rm (3)] If $ab=0$, then $\sigma^{\theta}(a)\delta^{\beta}(b)= \delta^{\beta}(a)\sigma^{\theta}(b) = 0$, for every $\theta, \beta\in \mathbb{N}^{n}$.
\end{enumerate}
\end{proposition}

If $R$ is $\Sigma$-rigid, it follows that $R$ is $(\Sigma, \Delta)$-compatible \cite[Proposition 3.4]{ReyesSuarez2018UMA}, but the converse does not hold \cite[Example 3.6] {ReyesSuarez2018UMA}. If $R$ is reduced, then $\Sigma$-rigid and $(\Sigma, \Delta)$-compatible coincide and if $A=\sigma(R)\langle x_1,\dots,x_n\rangle$ is a SPBW extension of $R$ then $A$ is reduced \cite[Theorem 3.9]{ReyesSuarez2018UMA}.

We present some examples of SPBW extensions over compatible rings.

\begin{example} 
 Commutative polynomial rings, PBW extensions, the algebra of linear partial differential operators, the algebra of linear partial shift operators, the algebra of linear partial difference operators, and the algebra of linear partial $q$-differential operators, and the class of diffusion algebras (for a detailed description, see \cite{Fajardoetal2020, Higuera2020, HigueraReyes2022}). Other examples include: Weyl algebras, multiplicative analogue of the Weyl algebra, some quantum Weyl algebras as $A_2(J_{a,b})$, the quantum algebra $U'(\mathfrak{so}(3, \Bbbk))$, the family of 3-dimensional skew polynomial algebras; Woronowicz algebra $W_v(\mathfrak{sl}(2, \Bbbk))$, the complex algebra $V_q({\mathfrak{sl}}_3(C))$, the Hayashi algebra $W_q(J)$, $q$-Heisenberg algebra ${\bf H}_n(q)$, and others.
\end{example}

\section{Skew PBW extensions having Property \((a.c.)\)}\label{MainResults}

In this section, we investigate SPBW extensions that satisfy Property \((a.c.)\). Hashemi et al.~\cite{HashemiMoussavi2005} explored a generalization of $\sigma$-rigid rings through the condition (SQA1), which represents a version of quasi-Armendariz rings for skew polynomial rings. We say that $R$ satisfies the (SQA1) condition if whenever $f(x)R[x;\sigma,\delta]g(x) = 0$ for $f(x) = \sum_{i=0}^{m} a_i x^i$ and $g(x) = \sum_{j=0}^{m} b_j x^j \in R[x;\sigma,\delta]$, then $a_i R b_j = 0$ for all $i,j$. Notice that if $R$ is $\sigma$-rigid, then $R$ also satisfies (SQA1). In \cite{ReyesSuarez2018UMA}, the notions (SA1) is defined for SPBW extensions with the aim of generalizing the results about Baer, quasi-Baer, p.p. and p.q.-Baer rings for $\Sigma$-rigid rings to the $(\Sigma,\Delta)$-compatible.

\begin{definition}[{\cite[Definition 4.1]{ReyesSuarez2018UMA}}]
Let \(\sigma(R)\langle
x_1,\dots,x_n\rangle\) be an SPBW extension of $R$. We say that $R$ satisfies the condition ${\rm (}SA1{\rm )}$ if whenever $fg = 0$ for $f = a_0 + a_1X_1 + \cdots + a_mX_m$ and
 $g = b_0 +b_1Y_1 +\cdots +b_tY_t$ elements of \(\sigma(R)\langle
x_1,\dots,x_n\rangle\), then $a_ib_j = 0$, for every $i,j$.
\end{definition}

Recall that a ring $R$ is called a \textit{Baer ring} if the right annihilator of every nonempty subset of $R$ is generated by an idempotent. Equivalently, for any nonempty subset $X \subseteq R$, there exists $e \in R$ with $e^2=e$ such that 
\(r_R(X)=eR.\). Under certain compatibility 
conditions, the following theorem shows how the Baer property implies the annihilator condition in skew PBW extensions.

\begin{theorem} \label{BaerSA1Theorem}
If \(R\) is a Baer and $(\Sigma,\Delta)$-compatible ring which satisfies the condition {\rm (}SA1{\rm )}, then \(\sigma(R)\langle
x_1,\dots,x_n\rangle\) has Property \((a.c.)\) on the right.
\end{theorem}

\begin{proof} If $A:=\sigma(R)\langle x_1,\dots,x_n\rangle$ is an SPBW extension of a Baer ring $R$ which satisfies the condition {\rm (}SA1{\rm )}, then $A$ is a Baer ring by \cite[Theorem 4.2]{ReyesSuarez2018UMA}, and hence $A$ has Property \((a.c.)\) on the right by \cite[Proposition 2.1]{MalkiLouzari2024}.
\end{proof}

A ring $R$ is called a \textit{right p.p.-ring} if the right annihilator of every element is generated by an idempotent. 
In other words, for each $a \in R$, there exists $e \in R$ with $e^2=e$ such that 
\(r_R(a)=eR.\) 
Similarly, $R$ is a \textit{left p.p.-ring} if the left annihilator of every element is generated by an idempotent. 
Right and left p.p.-rings are natural generalizations of Baer rings, 
because in Baer rings the annihilator of any subset (not just single elements) is generated by an idempotent element of $R$. 
For instance, any commutative domain is trivially a p.p.-ring because the annihilator of a non-zero element is zero, 
which is generated by the idempotent $0$. The following theorem shows that for an Abelian and NI ring, the p.p.-property guarantees that the SPBW extension has Property \((a.c.)\).

\begin{theorem}\label{AbelianNIppTheorem}
Let $R$ be a $(\Sigma,\Delta)$-compatible ring which is Abelian and NI, which
 satisfies {\rm (}SA1{\rm )}. If $R$ is a right {\rm (}resp., left{\rm )} p.p.-ring, then \(\sigma(R)\langle
x_1,\dots,x_n\rangle\) has Property \((a.c.)\) on the right {\rm (}resp., left{\rm )}.
\end{theorem}

\begin{proof}
Set $A:=\sigma(R)\langle x_1,\dots,x_n\rangle$. 
Since $R$ is $(\Sigma,\Delta)$-compatible, Abelian and NI, and $R$ is a right p.p.-ring, it follows from \cite[Theorem 4.7]{ReyesSuarez2018UMA} that $A$ is 
a right p.p.-ring, and $A$ is Abelian by \cite[Corollary 3.9]{Chacon} Every right p.p.-ring Abelian has Property $(a.c.) $ on the right by \cite[Proposition 2.3]{MalkiLouzari2024}, hence $A$ has Property $(a.c.)$ on the right.
\end{proof}

\begin{theorem}\label{BijectiveSPBWppTheorem}
Let \(\sigma(R)\langle
x_1,\dots,x_n\rangle\) be a bijective SPBW of a $\Sigma$-rigid, Abelian and NI ring. If $R$ is a right {\rm (}resp., left{\rm )} p.p.-ring, then \(\sigma(R)\langle
x_1,\dots,x_n\rangle\) has Property \((a.c.)\) on the right {\rm (}resp., left{\rm )}.
\end{theorem}

\begin{proof}
Let $A:=\sigma(R)\langle x_1,\dots,x_n\rangle$ be a bijective skew PBW extension of a 
$\Sigma$-rigid, Abelian and NI ring $R$. If $R$ is right (resp., left) p.p.-ring, then by 
\cite[Theorem 4.7]{ReyesSuarez2018UMA} $A$ is a right p.p.-ring and $A$ is an Abelian ring by \cite[Corollary 3.9]{Chacon}. By \cite[Proposition 2.3]{MalkiLouzari2024}, therefore $A$ has 
Property $(a.c.)$ on the right.
\end{proof}

A ring $R$ is called a \textit{right principally quasi-Baer ring} (right p.q.-Baer) 
if the right annihilator of every principal right ideal is generated by an idempotent. 
That is, for each $a \in R$, there exists $e \in R$ with $e^2=e$ such that 
\(r_R(aR)=eR.\) 
Similarly, $R$ is \textit{left p.q.-Baer} if the left annihilator of every principal left ideal 
is generated by an idempotent. 
Right and left p.q.-Baer rings generalize p.p.-rings: every p.p.-ring is automatically p.q.-Baer, 
but the converse is not true in general. 
For example, certain matrix rings over commutative rings can be p.q.-Baer without being p.p.-rings.
The next theorem establishes that a bijective SPBW extension over a 
$(\Sigma,\Delta)$-compatible p.q.-Baer ring 
satisfies Property \((a.c.)\).

\begin{theorem}\label{pqBaerTheorem}
If \(\sigma(R)\langle
x_1,\dots,x_n\rangle\) is bijective, and \(R\) is a right {\rm (}resp., left{\rm )} p.q.-Baer ring and $(\Sigma,\Delta)$-compatible, then \(\sigma(R)\langle
x_1,\dots,x_n\rangle\) has Property \((a.c.)\) on the right {\rm (}resp., left{\rm )}.
\end{theorem}

\begin{proof} Let $A:=\sigma(R)\langle x_1,\dots,x_n\rangle$ be a bijective SPBW extension of $R$. If $R$ is right {\rm (}resp., left{\rm )} p.q.-Baer and $(\Sigma,\Delta)$-compatible, then $A$ is right {\rm (}resp., left{\rm )} p.q.-Baer by \cite[Theorem 4.15]{ReyesSuarez2018UMA}, and hence $A$ has Property \((a.c.)\) on the right {\rm (}resp., left{\rm )} by \cite[Proposition 2.1]{MalkiLouzari2024}.
\end{proof}

\section{Examples}\label{Examplespaper}
The importance of our results is appreciated when we extend their application to algebraic structures that are more general than those considered by Malki and Louzari, that is, some noncommutative rings which cannot be expressed as skew polynomial rings $R[x;\sigma]$. In this section, we consider several families of rings that have been studied in the literature which are subfamilies of skew PBW extensions. Of course, the list of examples is not exhaustive.

\begin{example}
Algebras whose generators satisfy quadratic relations such as Clifford algebras, Weyl-Heisenberg algebras, and Sklyanin algebras, play an important role in analysis and mathematical physics. Motivated by these facts, Golovashkin and Maximov \cite{GolovashkinMaximov2005} considered the algebras $Q(a, b, c)$, with two generators $x$ and $y$, defined by the quadratic relations
\begin{equation}\label{GolovashkinMaximov2005(1)}
    yx = ax^2 + bxy + cy^2,
\end{equation}
    
where the coefficients $a$, $b$, and $c$ belong to an arbitrary field $\Bbbk$ of characteristic zero. They presented conditions on these elements under which such an algebra has a PBW basis of the form $\mathcal{B}=\{x^m y^n \mid m, n \in \mathbb{N}\}$. Golovashkin and Maximov given a necessary and sufficient condition for a PBW basis as above to exist when $ac + b \neq 0$ \cite[Section 3]{GolovashkinMaximov2005}. If $ac + b = 0$, they proved that if $b \neq 0, -1$ then $Q(a, b, c)$ has a PBW basis, and if $b = -1$ then the set $\mathcal{B}$ is linearly independent but do not form a PBW basis of $Q(a, b, c)$ \cite[Section 5]{GolovashkinMaximov2005}.
    
We can prove that if $a$, $b$, and $c$ are not all zero, then $Q(a, b, c)$ is neither a skew polynomial ring of $\Bbbk$, $\Bbbk[x]$, nor of $\Bbbk[y]$. If $b \neq 0$ and $c = 0$, one can verify that $Q(a, b, c)$ is an SPBW extension of $\Bbbk[x]$, that is $Q(a, b, c) \cong \sigma(\Bbbk[x])\langle y\rangle$. Theorem~\ref{AbelianNIppTheorem} establishes that the p.p.-property on the right 
(resp., left) guarantees that $Q(a, b, c)$ 
satisfies Property $(a.c.)$. When the extension is bijective, Theorem~\ref{BijectiveSPBWppTheorem} 
shows that this property is preserved for $Q(a, b, c)$. 
Furthermore, Theorem~\ref{pqBaerTheorem} proves that the same holds 
whenever $\Bbbk[x]$ is a $(\Sigma,\Delta)$-compatible right (resp., left) 
p.q.-Baer ring.

\end{example}
    
\begin{example}
Jordan \cite{Jordan2000} introduced a class of iterated skew polynomial rings $R(B, \sigma, c, p)$, known as \emph{ambiskew polynomial rings}, which include different examples of noncommutative algebras. These structures have been investigated by Jordan at various levels of generality in several papers \cite{Jordan1993, Jordan1993b, Jordan1995, Jordan1995b}. We recall the treatment of ambiskew polynomial rings in the framework appropriate to down-up algebras with $\beta \neq 0$.

Let $B$ be a commutative $\Bbbk$-algebra, $\sigma$ be a $\Bbbk$-automorphism of $B$, and elements $c \in B$ and $p \in \Bbbk^*$. If $S$ is the skew polynomial ring $B[x; \sigma^{-1}]$, and extend $\sigma$ to $S$ by setting $\sigma(x) = p x$, then there is a $\sigma$-derivation $\delta$ of $S$ such that $\delta(B) = 0$ and $\delta(x) = c$ \cite[p.~41]{Cohn1985}. The \emph{ambiskew polynomial ring} $R = R(B, \sigma, c, p)$ is the ring $S[y; \sigma, \delta]$, which satisfies the following relations:
\begin{equation}
    yx - p xy = c, \quad \text{and, for all } b \in B, \quad x b = \sigma^{-1}(b) x \quad \text{and} \quad y b = \sigma(b) y.
\end{equation}

Equivalently, $R$ can be presented as $R = B[y; \sigma][x; \sigma^{-1}, \delta']$, where $\sigma(y) = p^{-1} y$, $\delta'(B) = 0$, and $\delta'(y) = -p^{-1} c$, yielding the relation $xy - p^{-1} yx = -p^{-1} c$. If we interpret the relation $x b = \sigma^{-1}(b) x$ as $b x = x \sigma(b)$, we observe that the definition involves twists from both sides using $\sigma$; this motivates the name of these objects. It is well-known that every generalized Weyl algebra is isomorphic to a factor of an ambiskew polynomial ring, and the ring $R(B, \sigma, c, p)$ is isomorphic to the generalized Weyl algebra $B[w](\sigma, w)$, where $\sigma$ is extended to $B[w]$ by setting $\sigma(w) = p w + \sigma(c)$. It is not difficult to show that ambiskew polynomial rings are skew PBW extensions over $B$, that is $R(B, \sigma, c, p) \cong \sigma(B)\langle y, x \rangle$. If $R$ is Baer and $(\Sigma,\Delta)$-compatible satisfying condition (SA1), 
Theorem~\ref{BaerSA1Theorem} ensures that $R(B, \sigma, c, p)$ has Property (a.c.) on the right. Moreover, Theorem~\ref{AbelianNIppTheorem} shows that for an Abelian and NI ring, 
the right (resp., left) p.p.-property of $R$ passes to $R(B, \sigma, c, p)$ with 
respect to Property (a.c.). When the extension is bijective and $R$ is 
$\Sigma$-rigid, Abelian and NI, Theorem~\ref{BijectiveSPBWppTheorem} 
proves that Property $(a.c.)$ still holds on the right (resp., left). 
In addition, Theorem~\ref{pqBaerTheorem} establishes that if the extension is bijective 
and $R$ is a $(\Sigma,\Delta)$-compatible right (resp., left) p.q.-Baer ring, 
then Property $(a.c.)$ on the right (resp., left).
\end{example}

\begin{example}
{\em Skew bi-quadratic algebras} were introduced by Bavula \cite{Bavula2023} as a framework for classifying bi-quadratic algebras on three generators having a PBW basis.

If $L_n(R)$ denotes the set lower triangular matrices with elements of $R$, $C(R)$ is the center of $R$, $\sigma = (\sigma_1, \dotsc, \sigma_n)$ is an $n$-tuple of commuting endomorphisms of $R$, $\delta = (\delta_1, \dotsc, \delta_n)$ is an $n$-tuple of $\sigma$-derivations of $R$ (i.e., $\delta_i$ is a $\sigma_i$-derivation of $R$ for all $ i \le n$), $Q = (q_{ij}) \in L_n(Z(R))$, $\mathbb{A}:= (a_{ij, k}) \in L_n(R)$ and $\mathbb{B}:= (b_{ij}) \in L_n(R)$ where $1 \leq j < i \leq n$ and $k \le n$, then {\em skew bi-quadratic algebra} ({\em SBQA}), which is denoted by $R[x_1,\dotsc, x_n;\sigma, \delta, Q, \mathbb{A}, \mathbb{B}]$ is the ring generated by $R$ and the elements $x_1, \dotsc, x_n$ subject to the relations:
\begin{align}
    x_i r &= \sigma_i(r)x_i + \delta_i(r),\quad \text{for all}\ 1 \le i \le n, \text{ and } r \in R, \label{Bavula2023(1)} \\
    x_i x_j - q_{ij} x_j x_i &= \sum_{k=1}^{n} a_{ij, k} x_k + b_{ij},\quad \text{for all } j < i.\label{Bavula2023(2)}
\end{align}

If $\sigma_i$ is the identity homomorphism of $R$ and $\delta_i$ is the zero derivation of $R$, for all $1 \le i \le n$, then $R[x_1, \dotsc, x_n; Q, \mathbb{A}, \mathbb{B}]$ is called a {\em bi-quadratic algebra} ({\em BQA}). A skew bi-quadratic algebra has a {\em PBW basis} if $A = \bigoplus_{\alpha \in \mathbb{N}^{n}} R x^{\alpha}$ where $x^{\alpha} = x_1^{\alpha_1} \dotsb x_n^{\alpha_n}$.
\end{example}

The definition of SBQA makes it clear that $R[x_1, \dotsc, x_n;\sigma, \delta, Q, \mathbb{A}, \mathbb{B}] \cong \sigma(R)\langle x_1,\dotsc, x_n\rangle$. Moreover, SPBW extensions are more general than skew bi-quadratic algebras, since they do not require the $\sigma$'s to commute (cf. \cite[Definition 2.1]{AcostaLezamaReyes2015}) nor impose any membership conditions on $q_{ij}$ and $b_{ij}$. In this way, if $R$ is Abelian and $(\Sigma,\Delta)$-compatible, 
Theorem~\ref{AbelianNIppTheorem} shows that if $R$ is a right (resp., left) p.p.-ring, 
then the skew PBW extension $R[x_1, \dotsc, x_n;\sigma, \delta, Q, \mathbb{A}, \mathbb{B}]$ has Property $(a.c.)$. 
In the case when the extension is bijective and $R$ is $\Sigma$-rigid, Abelian and NI, 
Theorem~\ref{BijectiveSPBWppTheorem} confirms that Property $(a.c.)$ still holds on the right (resp., left). 
Also, Theorem~\ref{pqBaerTheorem} states that if $R$ is a $(\Sigma,\Delta)$-compatible 
right (resp., left) p.q.-Baer ring, then $R[x_1, \dotsc, x_n;\sigma, \delta, Q, \mathbb{A}, \mathbb{B}]$ has 
Property $(a.c.)$.

\section{Declarations}

The three authors were supported by Dirección Ciencias Básicas, Universidad ECCI- sede Bogotá. All authors declare that they have no conflicts of interest. 


\end{document}